\documentclass[11pt]{article}
\usepackage{a4wide}
\usepackage{amsmath}
\usepackage{amsfonts}
\usepackage{amsthm}
\usepackage{amssymb}
\usepackage[pdftex]{hyperref}
\usepackage[all]{xy}
\usepackage{flafter}
\usepackage[utf8]{inputenc}
\usepackage{multirow}
\usepackage{enumitem}
\newcommand{\N}{{\mathbb N}}
\newcommand{\Z}{{\mathbb Z}}
\newcommand{\Q}{{\mathbb Q}}
\newcommand{\R}{{\mathbb R}}
\newcommand{\C}{{\mathbb C}}
\newcommand{\semi}{{\mathbb o}}
\DeclareMathOperator{\lie}{{\mathfrak g}}
\DeclareMathOperator{\lieH}{{\mathfrak h}}

\DeclareMathOperator{\Aut}{{\rm Aut}}

\DeclareMathOperator{\Aff}{{\rm Aff}}
\DeclareMathOperator{\aff}{{\rm aff}}
\DeclareMathOperator{\Endo}{Endo}
\DeclareMathOperator{\Dim}{{\rm dim}}
\DeclareMathOperator{\Tr}{{\rm Tr}}
\DeclareMathOperator{\Id}{Id}
\DeclareMathOperator{\dd}{\delta}
\DeclareMathOperator{\D}{{\mathfrak D}}
\DeclareMathOperator{\A}{{\mathfrak A}}
\DeclareMathOperator{\B}{{\mathfrak B}}
\DeclareMathOperator{\GL}{{\rm GL}}
\newtheorem{theorem}{Theorem}[section]
\newtheorem{proposition}[theorem]{Proposition}
\newtheorem{lemma}[theorem]{Lemma}
\newtheorem{corollary}[theorem]{Corollary}
\newtheorem{definition}[theorem]{Definition}
\newtheorem*{remark}{Remark}

\newcommand{\ds}{\displaystyle}
\newcommand{\ra}{\rightarrow}

\begin{document}
\title{\bf Nielsen zeta functions on infra-nilmanifolds up to dimension three}
\author{Gert-Jan Dugardein\\
KULeuven Kulak, E.\ Sabbelaan 53, B-8500 Kortrijk}
\date{\today}
\maketitle

\begin{abstract}
In \cite{dd13-2}, we have proved that every Nielsen zeta function for a self-map on an infra-nilmanifold is rational. In this paper, we will list all possible self-maps on infra-nilmanifolds up to dimension three and we will give the Nielsen zeta function for any of those maps, by using the techniques developed in \cite{dd13-2}.
\end{abstract}
\section{Nielsen theory and dynamical zeta functions}
When $X$ is a compact polyhedron and $f:X\to X$ is a continuous self-map, we can define two different homotopy-invariant integers that give us information about the fixed points of $f$. The first of these integers is the Lefschetz number $L(f)$ which is defined as \[ L(f)= \sum_{i=0}^{{\Dim}\;X} (-1)^i {\Tr} \left( f_{\ast,i} : H_i(X,\Q) \to H_i(X,\Q)\right).\]
It is a well known fact that, if $L(f)\neq 0$, any map homotopic to $f$ has at least one fixed point. The second integer is the Nielsen number $N(f)$, which is defined as the number of essential fixed point classes of a $f$. By definition, $N(f)$ will be a lower bound for the number of fixed points of any map homotopic to $f$. More information on both numbers can be found in e.g. \cite{brow71-1}, \cite{jian83-1} and \cite{kian89-1}.

\medskip

It is clear that $N(f)$ contains more information about the fixed points of $f$ than $L(f)$. Unfortunately, the Nielsen number is often very hard to compute, while the Lefschetz number is easier to compute. This sometimes makes the Lefschetz number the better alternative to consider. For infra-nilmanifolds, however, there exists a formula which allows us to compute the Nielsen number in a strictly algebraic way. This is the main reason why the infra-nilmanifolds are a good class for the study of Nielsen theory.

\medskip

Lefschetz and Nielsen numbers can be used to define the so-called dynamical zeta functions. The Lefschetz zeta function of $f$ was introduced by S.~Smale in \cite{smal67-1} and is given by \[ L_f(z)=\exp\left( \sum_{k=1}^{+\infty}\frac{L(f^k)}{k}z^k\right).\]In his paper, Smale also proved that the Lefschetz zeta function is always rational for self-maps on compact polyhedra. The Nielsen zeta function, on the other hand, was defined for the first time by A.~Fel'shtyn \cite{fels88-1,fp85-1}: \[ N_f(z)=\exp\left( \sum_{k=1}^{+\infty}\frac{N(f^k)}{k}z^k\right).\]Unlike the Lefschetz zeta function, the Nielsen zeta function does not have to be rational in general. For a thorough discussion concerning dynamical zeta functions, we refer to \cite{fels00-2}.
\medskip

In \cite{dd13-2}, we were able to prove that Nielsen zeta functions for self-maps on infra-nilmanifolds are always rational. The idea is to show that $N(f^k)$ equals $|L(f^k)|$ or equals $|L(f^k) - L(f^k_+)|$ where $f_+$ is a lift of $f$ to a 2-fold covering of the given infra-nilmanifold. This allows us to express every Nielsen zeta function on an infra-nilmanifold as a quotient of at most two Lefschetz zeta functions, from which the rationality of the Nielsen zeta function follows easily.

\medskip

In this paper we will list all possible self-maps on infra-nilmanifolds of dimension $1$, $ 2$ and $3$ up to a homotopy. For each of these maps, we will compute the Nielsen zeta functions. Since these zeta functions are homotopy-invariant, this paper will contain all possible Nielsen zeta functions on lower-dimensional infra-nilmanifolds.

\section{Infra-nilmanifolds}
Infra-nilmanifolds are modeled on a connected and simply connected nilpotent Lie group $G$. Given such a Lie group $G$, we can consider the group of affine transformations on $G$, which is defined as the semi-direct product $\Aff(G)= G\semi \Aut(G)$. This group has a very natural left action on $G$:
\[ \forall (g,\alpha)\in \Aff(G),\, \forall h \in G: \;\;^{(g,\alpha)}h= g \alpha(h).\]When $G$ is abelian, $G$ will be isomorphic to $\R^n$ and $\Aff(G)$ will be the usual affine group.\medskip

Let $p:\Aff(G)=G\semi \Aut(G) \to \Aut(G)$ be the natural projection on the second factor.  

\begin{definition} When $G$ is a connected and simply connected nilpotent Lie group, a subgroup $\Gamma \subseteq \Aff(G)$ is called almost-crystallographic if and only if $p(\Gamma)$ is finite and $\Gamma\cap G$ is a uniform and discrete subgroup of $G$. The finite group $F=p(\Gamma)$ is called the holonomy group of $\Gamma$.
\end{definition}

Note that the action of $\Gamma$ on $G$ is properly discontinuous and cocompact. When $\Gamma$ is torsion-free, one can prove that this action is also free and therefore the resulting quotient space $\Gamma\backslash G$ will be a manifold. We have the following definition:
\begin{definition}
A torsion-free almost-crystallographic group $\Gamma\subseteq \Aff(G) $ 
is called an almost-Bieberbach group, and the corresponding manifold $\Gamma\backslash G$ is called an infra-nilmanifold (modeled on $G$).  
\end{definition} 

When $G$ is abelian, we will simply call $\Gamma$ a Bieberbach group and $\Gamma\backslash G$ will be a compact flat manifold.\medskip

The universal covering space of $\Gamma\backslash G$ will be the Lie group $G$. The fundamental group of this infra-nilmanifold will be $\Gamma$. When the holonomy group is trivial, then $\Gamma\cap G=\Gamma$. Hence, $\Gamma$ will be a lattice in $G$ and the corresponding manifold $\Gamma\backslash G$ is a nilmanifold. In the abelian case, this means that $\Gamma\cong \Z^n$ and that the corresponding manifold is a torus $T^n$.

\medskip

Before we can study Nielsen theory on infra-nilmanifolds, we first need to understand how all possible self-maps look. Of course, since $L(f)$ and $N(f)$ are homotopy-invariant, it suffices to know all these maps up to homotopy. In \cite{lee95-2}, K.B.\ Lee gives a complete description. Define the semigroup $\aff(G)=G\semi \Endo(G)$. Note that $\aff(G)$ acts on $G$ in a similar way as $\Aff(G)$. After all, for $\dd \in G$ and $\D\in \Endo(G)$, we can define the following action:\[ (\dd,\D): \; G \rightarrow G:\; h \mapsto \dd \D(h).\]Elements like $(\dd,\D)$ can be seen as affine maps, since $\aff(G)$ is a generalization of the semigroup of affine maps $\aff(\R^n)$ to the nilpotent case.

\begin{theorem}[K.B.\ Lee \cite{lee95-2}]
\label{leemaps} Let $G$ be a connected and simply connected nilpotent Lie group and suppose that $\Gamma, \Gamma'\subseteq \Aff(G)$ are two almost-crystallographic groups modeled on $G$. 
Then for any homomorphism $\varphi: \Gamma\rightarrow \Gamma'$ there 
exists an element $  (\dd, \D)\in \aff(G)$ such that 
\[ \forall \gamma \in \Gamma: \; \varphi(\gamma) (\dd, \D) =  (\dd, \D) \gamma.\] 
\end{theorem}

Note that we can consider the equality $ \varphi(\gamma) (\dd, \D) =  (\dd, \D) \gamma$ in $\aff(G)$, since $\Aff(G)$ is contained in $\aff(G)$. With this equality in mind, it is easy to see that the affine map $(\dd,\D)$ induces a well-defined map \[\overline{(\dd,\D)}: \Gamma \backslash G \rightarrow \Gamma' \backslash G: \; \Gamma h \rightarrow \Gamma' \dd \D(h),\]
which exactly induces the morphism $\varphi$ on the level of the fundamental groups.

\medskip

On the other hand, if we choose an arbitrary map $f:\Gamma\backslash G\ra \Gamma'\backslash G$ between two infra-nilmanifolds and choose a lifting $\tilde{f}:G \to G$ of $f$, then there exists a morphism $\tilde{f}_\ast:\Gamma\to \Gamma'$ such that $\tilde{f}_\ast(\gamma) \circ \tilde{f} = \tilde{f}\circ \gamma$, for all $\gamma\in \Gamma$. By Theorem~\ref{leemaps}, an affine map $(\dd,\D)\in \aff(G)$ exists which also satisfies $\tilde{f}_\ast(\gamma) \circ (\dd,\D)= (\dd,\D)\circ \gamma$ for all $\gamma\in \Gamma$. Therefore, the induced map $\overline{(\dd,D)} $ and 
$f$ are homotopic. We will call $(\dd,\D)$ an affine homotopy lift of $f$. 

\medskip 
In \cite{ll09-1}, J.B. Lee and K.B. Lee gave a formula to compute Lefschetz and Nielsen numbers on infra-nilmanifolds. Pick an infra-nilmanifold $\Gamma\backslash G$, determined by the almost-Bieberbach group $\Gamma\subseteq \Aff(G)$ and let $F\subseteq \Aut(G)$ denote the holonomy group of $\Gamma$. We will write $\lie$ for the Lie algebra of $G$. Because $G$ is a nilpotent, connected and simply connected Lie group, the map $\exp:\lie\to G$ will be a diffeomorphism. Therefore, $\Endo(G)$ and $\Endo(\lie)$ are isomorphic and for every endomorphism $\A\in \Endo(G)$, we have a unique $\A_\ast\in \Endo(\lie)$, which is determined by the relation $\A \circ \exp= \exp \circ \A_\ast$. This $\A_\ast$ will be called the differential of $\A$. Of course, $\A$ is invertible if and only if $\A_\ast$ is invertible.
\begin{theorem}[J.B.\ Lee and K.B.\ Lee \cite{ll09-1}] \label{LeeForm}Let $\Gamma\subseteq \Aff(G)$ be an almost-Bieberbach group with holonomy group $F\subseteq \Aut(G)$. Let $M=\Gamma\backslash G$ be the associated infra-nilmanifold.  If 
 $f:M\ra M$ is a map with affine homotopy lift $(\dd, \D)$, then  
\[L(f)=\frac{1}{\# F}\sum_{\A \in F}\det(I-\A_\ast\D_\ast)\]
and
\[N(f)=\frac{1}{\# F}\sum_{\A \in F}|\det(I-\A_\ast\D_\ast)|.\]
(Here $I$ is the identity).
\end{theorem}

\section{Preliminary results}
As we already mentioned in the first paragraph, the computations in this paper are largely based on the theory developed in \cite{dd13-2}. In this paragraph, we will repeat the necessary definitions and we will list the most important results.\medskip

To every almost-crystallographic group $\Gamma$ we can assign a representation, which describes the holonomy group of $\Gamma$.
\begin{definition}
For an almost-crystallographic group $\Gamma\subseteq \Aff(G)$, with holonomy group $F\subseteq  \Aut(G)$, the holonomy representation of $\Gamma$ is the representation 
\[ \rho: F \rightarrow \GL(\lie): \A \mapsto \A_\ast\]
\end{definition}
By choosing a basis of $\lie$, we can identify $\lie$ with $\R^n$ and this allows us to view the holonomy representation $\rho$ as a 
real representation $\rho: F \rightarrow \GL_n(\R)$. 

\medskip

By using this holonomy representation, we can give a necessary algebraic condition for an affine map $(\dd,\D)$ in order to be the affine homotopy lift of a map on the infra-nilmanifold. This condition will help us compute all possible self-maps on low-dimensional infra-nilmanifolds. 
\begin{proposition}[See \cite{ddp11-1}]\label{semi-affconj}
Let $\Gamma\subseteq \Aff(G)$ be an almost--Bieberbach group and let $M=\Gamma\backslash G$ be the corresponding infra--nilmanifold. Let $\rho:F \ra \GL(\lie)$ be the associated holonomy representation. If $f:M\ra M$ is a map with affine homotopy lift $(\dd,\D)$, there exists a function $\phi:F \ra F$ such that 
\[ \forall x \in F:\; \rho(\phi(x)) \D_\ast = \D_\ast \rho(x) .\]
\end{proposition}  

It might be tempting to believe that the function $\phi$ can always be chosen as a morphism of groups. However, as an example in \cite{ddp11-1} shows, this does not need to be the case.

\medskip

In \cite{dd13-2}, one can find the following proposition, which is a straightforward generalization of a result from \cite{ddp11-1}.

\begin{proposition}\label{decomprho}
Let $\rho:F\to \GL_n(\R)$ be a representation of a finite group $F$ and $\phi: F\to F$ be any function. Let $D\in \R^{n\times n}$. Suppose that $\rho(\phi(x))D=D\rho(x)$ for all $x\in F$. Then we can choose a basis of $\R^n$, such that $\rho=\rho_{\leq 1}\oplus \rho_{> 1}$, for representations $\rho_{\leq 1}:F\to \GL_{n_{\leq 1}}(\R)$ and $\rho_{> 1}:F\to \GL_{n_{> 1}}(\R)$ and such that $D$ can be written in block triangular form$$\left(\begin{array}{cc}
D_{\leq 1} & \ast \\
0 & D_{>1}
\end{array}\right),$$where $D_{\leq 1}$ and $D_{>1}$ only have eigenvalues of modulus $\leq 1$ and $>1$, respectively. 
\end{proposition}

Because of Proposition \ref{semi-affconj}, we know that when $(\dd,\D)$ is an affine homotopy lift of a map $f:M\to M$ on an infra-nilmanifolds, the conditions of Proposition \ref{decomprho} are met. This means that there exists a basis of $\R^n$, such that $\D_\ast$ has a block upper triangular form, while $\A_\ast$ has a block diagonal form, for every $\A\in F$, simultaneously. This decomposition of the holonomy representation will be called the decomposition of $\rho$ induced by $\D$.

\begin{theorem}[\cite{dd13-2}]\label{positivepart}
Let $M=\Gamma\backslash G$ be an infra-nilmanifold. Let $p:\Gamma\ra F$ denote the 
projection of $\Gamma$ onto its holonomy group and denote the holonomy representation by $\rho:F\to \GL_n(\lie)$. 
Take an arbitrary self-map $f:M\to M$ with $(\dd,\D)\in \aff(G)$ as affine homotopy lift. Let $\rho=\rho_{\leq 1} \oplus \rho_{>1}$ be the decomposition of $\rho$ induced by $\D$. Then the set
\[ \Gamma_+ = \{ \gamma \in \Gamma\;|\; \det(\rho_{>1}(p(\gamma))) =1 \}\]
is a normal subgroup of $\Gamma$ of index 1 or 2. Therefore, $\Gamma_+$ is also an almost-Bieberbach group and the corresponding infra-nilmanifold $M_+= \Gamma_+\backslash G$ is either equal to $M$ or to a 2-fold covering of $M$. In the second case, the map $f$ lifts to a map $f_+:M_+\to M_+$ with the same affine homotopy lift $(\dd,\D)$ as $f$.
\end{theorem}
When $\D_\ast$ has no eigenvalues of modulus $>1$, we will automatically assume that $\Gamma_+=\Gamma$.

From now on, $\Gamma_+$ will be called the positive part of $\Gamma$ with respect to $f$ and $f_+$ is the positive part of $f$.

\medskip
It is straightforward to see that for any $k\in \N$, $(\dd,\D)^k$ will be an affine homotopy lift of $f^k$. Therefore, the decomposition of $\rho$ into a direct sum $\rho=\rho_{\leq 1} \oplus \rho_{>1}$ is independent of the iteration $k$. Hence, the positive part $\Gamma_+$ of $\Gamma$ with respect to $f^k$ 
is the same as the positive part of $\Gamma$ with respect to $f^k$ for any $k\in \N$. Furthermore, also $(f_+)^k=(f^k)_+$.

\begin{theorem}[\cite{dd13-2}]\label{NielsenLefschetz}
Let $M=\Gamma\backslash G$ be an infra-nilmanifold and let $f:M\to M$ be a self-map with affine homotopy lift $(\dd,\D)$. Let $p$ denote the number of positive real eigenvalues of $\D_\ast$ which are strictly greater than $1$ and 
let $n$ denote the number of negative real eigenvalues of $\D_\ast$ which are strictly less than $-1$. Then we have the following table of equations:
\renewcommand{\arraystretch}{1.6}
\begin{center}\begin{tabular}{ |c|c|c| }
\cline{2-3}
\multicolumn{1}{c}{}
 &  \multicolumn{1}{|c|}{$\Gamma=\Gamma_+$}
 & \multicolumn{1}{|c|}{$\Gamma\neq \Gamma_+$} \\
\cline{1-3}
$k$ odd & $N(f^k)=(-1)^p L(f^k)$ & $N(f^k)=(-1)^{p} (L(f_+^k)-L(f^k))$ \\[1ex]
\cline{1-3}
$k$ even & $N(f^k)=(-1)^{p+n} L(f^k)$ & $N(f^k)=(-1)^{p+n} (L(f_+^k)-L(f^k))$ \\[1ex]
\cline{1-3}
\end{tabular}
 \end{center}
\end{theorem}

By using some straightforward calculations, we can immediately prove the following theorem, which is the main result for computing Nielsen zeta functions further in this paper. 

\begin{theorem}[\cite{dd13-2}]\label{Nielsenzeta}
Let $M=\Gamma\backslash G$ be an infra-nilmanifold and let $f:M\to M$ be a self-map with affine homotopy lift $(\dd,\D)$. Let $p$ denote the number of positive real eigenvalues of $\D_\ast$ which are strictly greater than $>1$ and 
let $n$ denote the number of negative real eigenvalues of $\D_\ast$ which are strictly less than $-1$. Then we have the following table of equations:
\renewcommand{\arraystretch}{1.9}
\begin{center}\begin{tabular}{ |c|c|c|c|c| }
\cline{2-5}
\multicolumn{1}{l}{ }
 &  \multicolumn{1}{|c|}{$p$ even, $n$ even}
 & \multicolumn{1}{|c|}{$p$ even, $n$ odd}& \multicolumn{1}{|c|}{$p$ odd, $n$ even}& \multicolumn{1}{|c|}{$p$ odd, $n$ odd} \\
\cline{1-5}
$\Gamma= \Gamma_+$ & $N_f(z)= L_f(z) $& $N_f(z)= \ds \frac{1}{L_{f}(-z)} $ &$N_f(z)= \ds \frac{1}{L_{f}(z)}$ & $N_f(z)= L_f(-z) $\\[1ex]
\cline{1-5}
$\Gamma\neq \Gamma_+$ &$N_f(z)= \ds \frac{L_{f_+}(z)}{L_{f}(z)}$ & $N_f(z)= \ds \frac{L_{f}(-z)}{L_{f_+}(-z)}$ &$N_f(z)= \ds \frac{L_{f}(z)}{L_{f_+}(z)}$ & $N_f(z)= \ds \frac{L_{f_+}(-z)}{L_{f}(-z)}$ \\[1ex]
\cline{1-5}
\end{tabular}
 \end{center}
 
\end{theorem}
Note that the case $\Gamma=\Gamma_+$ was proved independently in \cite{fl13-1}. It can be easily derived from Theorem \ref{NielsenLefschetz}, that this case coincides with the equality $N(f)=|L(f)|$, which was exactly the assumption that was made in \cite{fl13-1}. When this assumption holds, we say that $f$ meets the Anosov relation. If we can find conditions that tell us whether the Anosov relation holds, then this will decrease the number of calculations we need to do, since in this case we do not have to consider $\Gamma\neq \Gamma_+$. 

\medskip

An immediate corollary of Theorem \ref{positivepart} and Theorem \ref{NielsenLefschetz} is the following.
\begin{corollary}\label{AnCor}
Let $M=\Gamma\backslash G$ be an infra-nilmanifold and let $f:M\to M$ be a self-map with affine homotopy lift $(\dd,\D)$. Denote the holonomy representation by $\rho$. Let $\rho=\rho_{\leq 1} \oplus \rho_{>1}$ be the decomposition of $\rho$ induced by $\D$. If $\rho_{>1}$ is the identity representation or if $\rho=\rho_{\leq 1}$, then $N(f)=|L(f)|$.
\end{corollary}

The following is one of the most famous results in Nielsen theory. It was proved separately by Anosov \cite{anos85-1} and Fadell and Husseini \cite{fh86-1}. 

\begin{theorem}\label{AnNil}
Let $f: N\to N$ be a self-map on a nilmanifold, then $N(f)=|L(f)|$.
\end{theorem}

The following two results are more recent and give certain conditions under which the Anosov relation holds for infra-nilmanifolds.

\begin{theorem}[Dekimpe, De Rock, Malfait \cite{ddm04-3}]\label{AnCycl}
Let $f:M\to M$ be a self-map on an infra-nilmanifold. Suppose that the holonomy group $F$ of $M$ is cyclic, with $x_0\in F$ as generator and $\rho: F\to \GL_n(\lie)$ as holonomy representation. If $\rho(x_0)$ has no eigenvalues equal to $-1$, then $N(f)=|L(f)|$.
\end{theorem}

\begin{theorem}[Dekimpe, De Rock, Penninckx \cite{ddp11-1}]\label{An2P}
Let $F$ be a finite group. Then the Anosov relation holds for every self-map on any infra-nilmanifold with holonomy group $F$ if and only if $F$ has no index-two subgroups.
\end{theorem}
\section{The theoretical idea behind the computations}
In order to obtain all possible Lefschetz and Nielsen zeta functions for self-maps on infra-nilmanifolds, we have to know what the possibilities are for such  self-maps. By the work of K.B. Lee (\cite{lee95-2}), we know that every such a map is homotopic to a map that is induced by an affine map on the Lie group. Because $N(f)$ and $L(f)$ are topological invariants, we know it suffices to know all possible affine maps that induce a self-map on the infra-nilmanifold. In a certain sense, calculating which maps are possible, is the hard part in calculating these zeta functions.

\medskip

For infra-nilmanifolds up to dimension $3$, there are only two possibilities for the nilpotent Lie group $G$ on which the manifold is modeled. It can either be modeled on $\R^n$, which coincides with the abelian case, or it can be modeled on the $2$-step nilpotent Heisenberg group.

\subsection{Flat manifolds, the abelian case}

\subsubsection{The practical approach}
Let $\D\in \R^{n\times n}$ and $\dd\in \R^n$ and let $\Gamma$ be a Bieberbach group. Then $(\dd,\D)$ will be a self-map on the infra-nilmanifold $\Gamma\backslash\R^n$, if and only if there exists a morphism $\varphi:\Gamma\to \Gamma$, such that for every $(\alpha,\A)\in \Gamma$ \begin{equation}\label{keyequ}
(\dd,\D)(\alpha,\A)=\varphi(\alpha,\A)(\dd,\D).
\end{equation}

For computational reasons, we will always assume $\Gamma\cap \R^n$ to be precisely $\Z^n$. This can always be achieved by making an affine change of coordinates. Under this assumption, we know that two elements of $\Gamma$ with the same rotational part can only differ by an integer vector in their translational part.

\medskip

To check for which affine maps the condition (\ref{keyequ}) is met, we pick a set of generators for $\Gamma$, preferably as small as possible. To every generator $(\alpha,\A)$ we assign a $(\beta,\B)\in \Gamma$ and check under which assumptions the relation $$(\dd,\D)(\alpha,\A)=(\beta,\B)(\dd,\D)$$is satisfied. 

\medskip

Practically, we assign an arbitrary element $\B$ of $F$ to every of the chosen generators $(\alpha,\A)$. We check which conditions need to hold for $\D$ in order to have that $\D\A=\B\D$, for all generators simultaneously. By doing this, we actually check for conditions on $\D$ in order to have an equality for the rotational parts of equation (\ref{keyequ}). Most of the time this heavily simplifies $\D$, which makes further computations a lot easier.

\medskip

Subsequently, we compare the translational parts of equation (\ref{keyequ}). Since the lattice of our group equals $\Z^n$, we only have to check under which conditions $$\dd+\D(\alpha)\in \beta+\B(\dd) +\Z^n.$$When no contradictions occur, the conditions that are found will give us an affine homotopy lift of a map on our infra-nilmanifold. Since our method takes all possible morphisms $\varphi:\Gamma\to \Gamma$ into account, we will find all possible self-maps on infra-nilmanifolds up to a homotopy.

\subsubsection{Example: the Klein bottle}
To exemplify our approach, we will compute all possible maps on the Klein bottle.
\medskip

First of all, note that the following semigroup representation is faithful:$$\aff(\R^n)\to M_{n+1}(\R):(\dd,\D)\mapsto \begin{pmatrix}
\D & \dd\\
0 & 1
\end{pmatrix}.$$This means that it is possible to embed $\aff(\R^n)$ in $M_{n+1}(\R)$ in such a way that the operation on $\aff(\R^n)$ becomes the matrix multiplication. For computational reasons, we will consider all the elements of our Bieberbach groups to be matrices according to this embedding.\medskip

The Klein bottle group $\Gamma_K$ has the following generators:
$$\begin{pmatrix}
1 & 0&\frac{1}{2}\\
0 & -1&\frac{1}{2}\\
0 & 0&1
\end{pmatrix}\textrm{ and } \begin{pmatrix}
1 & 0&0\\
0 & 1&1\\
0 & 0&1
\end{pmatrix}.$$

Notice that $$\begin{pmatrix}
1 & 0&\frac{1}{2}\\
0 & -1&\frac{1}{2}\\
0 & 0&1
\end{pmatrix}^2=\begin{pmatrix}
1 & 0&1\\
0 & 1&0\\
0 & 0&1
\end{pmatrix}.$$Therefore, $\Gamma_K\cap \R^2=\Z^2$ and our generators are already written down in the desired form. Let us start from a general affine map on $\R^2$:$$\begin{pmatrix}
a & b & r\\
c & d & s\\
0& 0 & 1
\end{pmatrix},$$with $a,b,c,d,r,s \in \R$. First we need to check which conditions arise from looking at the relation $\D\A=\B\D$, for every generator $(\alpha, \A)$ of $\Gamma_K$ simultaneously. After this, we only need to check which extra conditions arise from the relation $\dd+\D(\alpha)\in \beta+\B(\dd) +\Z^2$. Because $\Gamma_K$ has two generators and because there are exactly two possibilities for the rotational parts of elements in $\Gamma_K$, we only have to distinguish between $4$ different cases. 
\medskip

\underline{Case 1:} $$\begin{pmatrix}
a&b\\
c&d
\end{pmatrix}\begin{pmatrix}
1&0\\
0&-1
\end{pmatrix}=\begin{pmatrix}
1&0\\
0&-1
\end{pmatrix}\begin{pmatrix}
a&b\\
c&d
\end{pmatrix} \textrm{ and } \begin{pmatrix}
a&b\\
c&d
\end{pmatrix}\begin{pmatrix}
1&0\\
0&1
\end{pmatrix}=\begin{pmatrix}
1&0\\
0&-1
\end{pmatrix}\begin{pmatrix}
a&b\\
c&d
\end{pmatrix}.$$
We find that $b=c=d=0$ and hence, the rotational part of our affine map is$$\begin{pmatrix}
a &0\\
0&0
\end{pmatrix}.$$
If we now include the translational parts, we find$$\begin{pmatrix}
\frac{a}{2}+r\\
s
\end{pmatrix}\in \begin{pmatrix}
\frac{1}{2}+r\\
\frac{1}{2}-s
\end{pmatrix}+\Z^2 \textrm{ and } \begin{pmatrix}
r\\
s
\end{pmatrix}\in \begin{pmatrix}
\frac{1}{2}+r\\
\frac{1}{2}-s
\end{pmatrix}+\Z^2.$$

This leads to a contradiction, since it is impossible that $r=\frac{1}{2}+r+z$ for a $z\in \Z$.

\medskip

\underline{Case 2:} $$\begin{pmatrix}
a&b\\
c&d
\end{pmatrix}\begin{pmatrix}
1&0\\
0&-1
\end{pmatrix}=\begin{pmatrix}
1&0\\
0&-1
\end{pmatrix}\begin{pmatrix}
a&b\\
c&d
\end{pmatrix} \textrm{ and } \begin{pmatrix}
a&b\\
c&d
\end{pmatrix}\begin{pmatrix}
1&0\\
0&1
\end{pmatrix}=\begin{pmatrix}
1&0\\
0&1
\end{pmatrix}\begin{pmatrix}
a&b\\
c&d
\end{pmatrix}.$$
We find that $b=c=0$ and hence, the rotational part of our affine map is$$\begin{pmatrix}
a &0\\
0&d
\end{pmatrix}.$$
By including the translational parts, we find

$$\begin{pmatrix}
\frac{a}{2}+r\\
\frac{-d}{2}+s
\end{pmatrix}\in \begin{pmatrix}
\frac{1}{2}+r\\
\frac{1}{2}-s
\end{pmatrix}+\Z^2 \textrm{ and } \begin{pmatrix}
r\\
d+s
\end{pmatrix}\in \begin{pmatrix}
r\\
s+1
\end{pmatrix}+\Z^2.$$
Some small calculations lead us to the following two possibilities for extra conditions on the affine map:
\begin{itemize}
\item $a$ and $d$ are odd integers, $s\in \frac{1}{2}\Z$.
\item $a$ is odd, $d$ is even and $s\in \frac{1}{4}\Z\backslash \frac{1}{2}\Z$. 
\end{itemize}
\medskip

\underline{Case 3:} $$\begin{pmatrix}
a&b\\
c&d
\end{pmatrix}\begin{pmatrix}
1&0\\
0&-1
\end{pmatrix}=\begin{pmatrix}
1&0\\
0&1
\end{pmatrix}\begin{pmatrix}
a&b\\
c&d
\end{pmatrix} \textrm{ and } \begin{pmatrix}
a&b\\
c&d
\end{pmatrix}\begin{pmatrix}
1&0\\
0&1
\end{pmatrix}=\begin{pmatrix}
1&0\\
0&-1
\end{pmatrix}\begin{pmatrix}
a&b\\
c&d
\end{pmatrix}.$$
We find that $b=c=d=0$ and hence, the rotational part of our affine map is$$\begin{pmatrix}
a &0\\
0&0
\end{pmatrix}.$$

Note that this case will yield the exact same contradiction as in Case 1.
\medskip

\underline{Case 4:} $$\begin{pmatrix}
a&b\\
c&d
\end{pmatrix}\begin{pmatrix}
1&0\\
0&-1
\end{pmatrix}=\begin{pmatrix}
1&0\\
0&1
\end{pmatrix}\begin{pmatrix}
a&b\\
c&d
\end{pmatrix} \textrm{ and } \begin{pmatrix}
a&b\\
c&d
\end{pmatrix}\begin{pmatrix}
1&0\\
0&1
\end{pmatrix}=\begin{pmatrix}
1&0\\
0&1
\end{pmatrix}\begin{pmatrix}
a&b\\
c&d
\end{pmatrix}.$$
We find that $b=d=0$ and hence, the rotational part of our affine map is$$\begin{pmatrix}
a &0\\
c&0
\end{pmatrix}.$$
By using the translational parts, we find

$$\begin{pmatrix}
\frac{a}{2}+r\\
\frac{c}{2}+s
\end{pmatrix}\in \begin{pmatrix}
r\\
s+1
\end{pmatrix}+\Z^2 \textrm{ and } \begin{pmatrix}
r\\
s
\end{pmatrix}\in \begin{pmatrix}
r\\
s+1
\end{pmatrix}+\Z^2.$$

Therefore, we get the following extra conditions:
\begin{itemize}
\item $a$ and $c$ are even integers.
\end{itemize}

\subsection{Infra-nilmanifolds modeled on the three-dimensional Heisenberg group}
\subsubsection{The practical approach}

The following group, with matrix multiplication as operation, is generally called the three-dimensional Heisenberg group:$$H:=\left\{\left(\begin{matrix}
1&x&y\\
0&1&z\\
0&0&1
\end{matrix}\right) \| x,y,z\in \R \right\}.$$
For any integer $k\neq 0$, we can consider the following elements in $H$:$$a=\left(\begin{matrix}
1&0&0\\
0&1&1\\
0&0&1
\end{matrix}\right)\textrm{ } b=\left(\begin{matrix}
1&1&0\\
0&1&0\\
0&0&1
\end{matrix}\right)\textrm{ } c=\left(\begin{matrix}
1&0&\frac{1}{k}\\
0&1&0\\
0&0&1
\end{matrix}\right).$$
Note that they form a uniform lattice of $H$ and that $[b,a]=c^k$. For notational reasons, we will write $h(x,y,z)$ instead of the matrix $$\left(\begin{matrix}
1&y&\frac{z}{k}\\
0&1&x\\
0&0&1
\end{matrix}\right).$$It is easy to see that $h(1,0,0), h(0,1,0)$ and $h(0,0,1)$ coincide with $a, b$ and $c$, respectively.

\medskip

Consider a three-dimensional infra-nilmanifold $\Gamma\backslash H$ and let $\Id \in \Aut(H)$ be the identity automorphism. After conjugating with an element of $\Aff(H)$, we can always assume $\Gamma\cap H$ to be the group generated by the elements $(a, \Id), (b, \Id)$ and $(c,\Id)$. This means that we can always assume the lattice of $\Gamma$ to be consisting of the elements: $$\{(h(z_1,z_2,z_3), \Id)\| z_1,z_2,z_3\in \Z \}.$$

The following lemma will make our work a lot easier, since it provides us with a way to do all the computations in $\aff(H)$ by working on matrices. It is a generalized version of the representation found in \cite{ddp04-2}. 
\begin{lemma}
The following map is a faithful representation:$$\psi: H\rtimes \Endo(H)\to M_4(\R):(h(x,y,z),\varphi)\to \left(\begin{matrix}
1&\frac{k y}{2}&\frac{-kx}{2}& \frac{-kxy}{2}+z\\
0&1&0&x\\
0&0&1&y\\
0&0&0&1
\end{matrix}\right)\cdot M_{\varphi}.$$Here $M_\varphi$ denotes the following block matrix:$$\left(\begin{matrix}
\varphi_\ast&0\\
0&1
\end{matrix}\right),$$where $\varphi_\ast$ is the matrix representation of the differential of $\varphi$, with respect to the basis $\{\log(c),\log(a),\log(b)\}$.
\end{lemma}
We will omit the proof since it involves a long, but straightforward computation. 
\medskip

Note that this lemma actually tells us that the semigroup $H\rtimes \Endo(H)$ can be embedded in the semigroup $\aff(\R^3)$. By using this embedding and the fact that the lattice of $\Gamma$ is precisely the group $\{(h(z_1,z_2,z_3), \Id)\| z_1,z_2,z_3\in \Z \}$, it is possible to compute all possible affine maps in the exact same way as in the abelian case.

\begin{remark}
In \cite{deki94-2}, it is shown that such an embedding always exists for $2$-step nilpotent Lie groups.
\end{remark}

\subsection{Computation of the Lefschetz zeta function}

When we want to use Theorem \ref{Nielsenzeta} to compute the Nielsen zeta functions, we need to know how to calculate Lefschetz zeta functions. One of the tools for this computation, is the following lemma.

\begin{lemma}\label{lemzeta}
Let $f$ be a self-map on a compact polyhedron. If, for every $k>0$, we have that$$\ds L(f^k)=\sum_{i=1}^{n_a}a_i^k-\sum_{j=1}^{n_b}b_j^k,$$with $a_i,b_j \in \C$, then the Lefschetz zeta function $L_f(z)$ will be equal to $$L_f(z)=\frac{\ds\prod_{j=1}^{n_b}(1-b_jz)}{\ds\prod_{i=1}^{n_a}(1-a_iz)}.$$
\end{lemma} 

\begin{proof}
The following power series is well known:$$\sum_{k=1}^{\infty}\frac{a^kz^k}{k}=-\ln(1-az).$$Therefore, $$\sum_{k=1}^{\infty}\frac{L(f^k)z^k}{k}=\sum_{i=1}^{n_a}(-\ln(1-a_iz))-\sum_{j=1}^{n_b}(-\ln(1-b_jz)).$$The rest of the proof is straightforward.

\end{proof}

Because of the definition of the Lefschetz number, we know it is always possible to express $L(f^k)$ as a sum of $k$-th powers. Indeed, since every infra-nilmanifold $M$ is a compact polyhedron, we know that $H_i(M,\Q)$ will be a finite-dimensional vector space over $\Q$, for every $i=0,\dots \Dim M$. Therefore, \[ L(f^k)= \sum_{i=0}^{{\Dim}\;M} (-1)^i {\Tr} \left( f_{\ast,i}^k : H_i(M,\Q) \to H_i( M,\Q)\right)\]can be expressed as the alternating sum of the $k$-th powers of eigenvalues of the linear maps $f_{\ast,i}$. Hence, we can always find complex numbers such that the Lefschetz zeta function looks like the resulting zeta function in Lemma \ref{lemzeta}.

\medskip

Although this argumentation provides us with the existence of numbers such that $$\ds L(f^k)=\sum_{i=1}^{n_a}a_i^k-\sum_{j=1}^{n_b}b_j^k,$$it can be rather complicated to find these numbers explicitly. We therefore chose to work with the formula of Theorem \ref{LeeForm}, from which we managed to compute $L(f^k)$ directly. We can always put $\D_\ast$ in Jordan normal form $\Lambda$, by conjugating with an invertible matrix $P$. In this way, we find$$\det(I-\A_\ast \D_\ast^k)=\det(I-P\A_\ast P^{-1} \Lambda^k).$$By using this argument, we can isolate the $k$-th powers of the eigenvalues of $\D_\ast$, which makes it easier to find an expression for $L(f^k)$ as a sum of $k$-th powers. 

\medskip

We give a small example. Take the map $(\dd,\D):\R^2\to \R^2$, with $$\D=\begin{pmatrix}
a &0\\
0 & b
\end{pmatrix} \textrm{ and } \dd=\begin{pmatrix}
0 \\
0
\end{pmatrix},$$with $a,b$ odd integers. This map is an affine homotopy lift of a continuous self-map on the Klein bottle. Due to the formula in Theorem \ref{LeeForm} and by using the fact that $\D$ is already written down in Jordan form, we see that$$L(f^k)=\frac{1}{2}\left((1-a^k)(1-b^k)+(1-a^k)(1+b^k)\right)=1-a^k.$$Lemma \ref{lemzeta} tells us that $$L_f(z)=\frac{1-az}{1-z}.$$

\section{Computational results}
\subsection{The abelian case}
In dimension $1$, there is only $1$ flat manifold, the circle. In dimension $2$, there are $2$: the $2$-torus and the Klein bottle. In dimension $3$, there exist $10$ different flat manifolds, among which the $3$-torus and the Hantzsche-Wendtmanifold are the most well known. For every of those manifolds, we can find the generators of the corresponding Bieberbach groups on \href{http://wwwb.math.rwth-aachen.de/carat/bieberbach.html}{the Page of Lower Dimensional Bieberbach Groups}, \cite{carat06-1}. 
\subsubsection{Tori}
It is well known and easy to check that every map on a torus is induced by an affine map with an integer-valued rotational part and an arbitrary translational part. 
\begin{proposition}[\cite{fels00-2}]
Let $f:T^n\to T^n$ be a self-map on a torus which is induced by the matrix $D$. Let $\lambda_1,\dots \lambda_n$ be the eigenvalues of $D$, let $\mathcal{P}(n)$ be the power set of $\{1,\dots,n\}$ and assume that $\prod_{j\in \emptyset}\lambda_j=1$, then $$L_f(z)=\prod_{A \in \mathcal{P}(n)}\left(1-\left(\prod_{j\in A}\lambda_j\right)z\right)^{{(-1)}^{\# A+1}}.$$
\end{proposition}

\begin{proof}
This follows easily from Lemma \ref{lemzeta} and the fact that$$L(f^k)=\det(I-D^k)=\prod_{j=1}^n(1-\lambda_j^k).$$
\end{proof}

Because the Anosov relation always holds on tori (e.g. due to Theorem \ref{AnNil} ), we know $\Gamma=\Gamma_+$. Then the previous proposition, together with Theorem \ref{Nielsenzeta}, gives us the following corollaries. From now on, $p$ will always denote the number of eigenvalues strictly greater than $1$, while $n$ denotes the number of eigenvalues strictly less than $-1$, like in Theorem \ref{Nielsenzeta}.

\begin{corollary}
If $f:S^1\to S^1$ is a map of degree $d$, then $N_f(z)$ can be found in the following diagram:
\renewcommand{\arraystretch}{1.8}
\begin{center}\begin{tabular}{ c|c|c}
  \multicolumn{1}{c|}{$d<-1$}
 & \multicolumn{1}{|c|}{$-1\leq d\leq 1$}& \multicolumn{1}{|c}{$1<d$} \\
\cline{1-3}
 $\ds \frac{1+z}{1+dz}$& $\ds\frac{1-dz}{1-z}$ &$\ds\frac{1-z}{1-dz}$ \\[1ex]

\end{tabular}
 \end{center}
\end{corollary}

\begin{corollary}
Let $f:T^2\to T^2$ be a self-map which is induced by the matrix $D$, with eigenvalues $\lambda_1,\lambda_2$. Then $N_f(z)$ can be found in the following diagram:
\renewcommand{\arraystretch}{2}
\begin{center}\begin{tabular}{ c|c|c|c }
 \multicolumn{1}{c|}{$p$ even, $n$ even}
 & \multicolumn{1}{|c|}{$p$ even, $n$ odd}& \multicolumn{1}{|c|}{$p$ odd, $n$ even}& \multicolumn{1}{|c}{$p$ odd, $n$ odd} \\
\cline{1-4}
$\ds\frac{(1-\lambda_1z)(1-\lambda_2z)}{(1-z)(1-\lambda_1\lambda_2z)} $& $\ds\frac{(1+z)(1+\lambda_1\lambda_2z)}{(1+\lambda_1z)(1+\lambda_2z)}  $ &$\ds\frac{(1-z)(1-\lambda_1\lambda_2z)}{(1-\lambda_1z)(1-\lambda_2z)} $ & $\ds\frac{(1+\lambda_1z)(1+\lambda_2z)}{(1+z)(1+\lambda_1\lambda_2z)} $\\[1ex]

\end{tabular}
 \end{center}
\end{corollary}
\begin{corollary}
Let $f:T^3\to T^3$ be a self-map which is induced by the matrix $D$, with eigenvalues $\lambda_1,\lambda_2, \lambda_3$. Then $N_f(z)$ can be found in the following diagram:
\renewcommand{\arraystretch}{2}
\begin{center}\begin{tabular}{ c|c|c }
 \multicolumn{1}{c|}{}& \multicolumn{1}{|c|}{$p$ even}
 & \multicolumn{1}{|c}{$p$ odd} \\
\cline{1-3}
$n$ even & $\frac{(1-\lambda_1z)(1-\lambda_2z)(1-\lambda_3z)(1-\lambda_1\lambda_2\lambda_3z)}{(1-z)(1-\lambda_1\lambda_2z)(1-\lambda_1\lambda_3z)(1-\lambda_2\lambda_3z)} $ &$\frac{(1-z)(1-\lambda_1\lambda_2z)(1-\lambda_1\lambda_3z)(1-\lambda_2\lambda_3z)}{(1-\lambda_1z)(1-\lambda_2z)(1-\lambda_3z)(1-\lambda_1\lambda_2\lambda_3z)}$ \\[1ex]
\cline{1-3}
$n$ odd & $\frac{(1+z)(1+\lambda_1\lambda_2z)(1+\lambda_1\lambda_3z)(1+\lambda_2\lambda_3z)}{(1+\lambda_1z)(1+\lambda_2z)(1+\lambda_3z)(1+\lambda_1\lambda_2\lambda_3z)}$ & $\frac{(1+\lambda_1z)(1+\lambda_2z)(1+\lambda_3z)(1+\lambda_1\lambda_2\lambda_3z)}{(1+z)(1+\lambda_1\lambda_2z)(1+\lambda_1\lambda_3z)(1+\lambda_2\lambda_3z)} $

\end{tabular}
 \end{center}
\end{corollary}
\subsubsection{Klein Bottle}
As already mentioned in a previous example, the Klein Bottle group has the following generators:

$$\left(\begin{matrix}
1 & 0&\frac{1}{2}\\
0 & -1&\frac{1}{2}\\
0 & 0&1
\end{matrix}\right)\textrm{ and } \left(\begin{matrix}
1 & 0&0\\
0 & 1&1\\
0 & 0&1
\end{matrix}\right).$$
The holonomy group is isomorphic to $\Z_2$.\\

\underline{Possible affine homotopy lifts:}
\begin{itemize}
\item $\left(\begin{matrix}
a & 0&r\\
0 & b&s\\
0 & 0&1
\end{matrix}\right)$, with $a$ and $b$ odd, $r\in \R$ and $s\in \frac{1}{2}\Z$.

\item $\left(\begin{matrix}
a & 0&r\\
0 & b&s\\
0 & 0&1
\end{matrix}\right)$, with $a$ odd, $b$ even, $r\in \R$ and $s\in \frac{1}{4}\Z\setminus \frac{1}{2}\Z$.

The Nielsen zeta functions can be found in the following table:

\begin{center}
\renewcommand{\arraystretch}{1.7}\begin{tabular}{ c|c|c|c|c }
\multicolumn{1}{l}{ }
 &  \multicolumn{1}{|c|}{$p$ even, $n$ even}
 & \multicolumn{1}{|c|}{$p$ even, $n$ odd}& \multicolumn{1}{|c|}{$p$ odd, $n$ even}& \multicolumn{1}{|c}{$p$ odd, $n$ odd} \\
\cline{1-5}
$\Gamma= \Gamma_+$ & $\ds\frac{1-az}{1-z}$ & $\ds\frac{1+z}{1+az}$ &$\ds\frac{1-z}{1-az}$ & $\ds\frac{1+az}{1+z}$\\[1ex]
\cline{1-5}
$\Gamma\neq \Gamma_+$ &$\ds\frac{1-bz}{1-abz}$ & $\ds\frac{1+abz}{1+bz}$ &$\ds\frac{1-abz}{1-bz}$ & $\ds\frac{1+bz}{1+abz}$\\[1ex]
\end{tabular}
 \end{center}
 
\item $\left(\begin{matrix}
a & 0&r\\
b & 0&s\\
0 & 0&1
\end{matrix}\right)$, with $a$ and $b$ even and $r,s\in \R$.

Because of Corollary \ref{AnCor}, we know that $\Gamma=\Gamma_+$ is the only possibility in this case. Hence, we have the following Nielsen zeta functions:

\begin{center}
\renewcommand{\arraystretch}{1.8}\begin{tabular}{ c|c|c|c|c }
\multicolumn{1}{l}{ }
 &  \multicolumn{1}{|c|}{$p$ even, $n$ even}
 & \multicolumn{1}{|c|}{$p$ even, $n$ odd}& \multicolumn{1}{|c|}{$p$ odd, $n$ even}& \multicolumn{1}{|c}{$p$ odd, $n$ odd} \\
\cline{1-5}
$\Gamma= \Gamma_+$ & $\ds\frac{1-az}{1-z}$ & $\ds\frac{1+z}{1+az}$ &$\ds\frac{1-z}{1-az}$ & $\ds\frac{1+az}{1+z}$\\[1ex]
\end{tabular}
 \end{center}
\end{itemize}

\subsubsection{Hantzsche-Wendt}
The Hanztsche-Wendtgroup has the following generators:

$$\left(\begin{matrix}
-1&0&0&0\\
0 & 1 & 0& \frac{1}{2}\\
0&0&-1 &\frac{1}{2}\\
0&0&0&1
\end{matrix}\right) \textrm{ and } \left(\begin{matrix}
1&0&0&\frac{1}{2}\\
0 & -1 & 0& 0\\
0&0&-1 &0\\
0&0&0&1
\end{matrix}\right) .$$
The holonomy group is isomorphic to $\Z_2\oplus \Z_2$.\\

\underline{Possible affine homotopy lifts:}
\begin{itemize}
\item $\left(\begin{matrix}
a&0&0&r\\
0 & b & 0& s\\
0&0&c &t\\
0&0&0&1
\end{matrix}\right)$, with $a,b$ and $c$ odd and $r,s,t \in \frac{1}{2}\Z$.

\begin{center}
\renewcommand{\arraystretch}{1.7}\begin{tabular}{ c|c|c|c|c }
\multicolumn{1}{l}{ }
 &  \multicolumn{1}{|c|}{$p$ even, $n$ even}
 & \multicolumn{1}{|c|}{$p$ even, $n$ odd}& \multicolumn{1}{|c|}{$p$ odd, $n$ even}& \multicolumn{1}{|c}{$p$ odd, $n$ odd} \\
\cline{1-5}
$ \Gamma= \Gamma_+$ & $\ds \frac{1-abcz}{1-z}$ & $\ds \frac{1+z}{1+abcz}$ &$\ds \frac{1-z}{1-abcz}$ & $\ds \frac{1+abcz}{1+z}$\\[1ex]
\end{tabular}
 \end{center}

Suppose now that $\Gamma\neq \Gamma_+$ and denote our affine map by $(\dd,\D)$. Then the subgroup $\Gamma_+$ may vary for different values of $a,b$ and $c$. Denote $\Lambda_{\leq 1}$ and $\Lambda_{>1}$ as all the eigenvalues of $\D$ with modulus $\leq 1$, $>1$, respectively.  When $a, b$ and $c$ are simultaneously in $\Lambda_{\leq 1}$, then $\Gamma=\Gamma_+$. The same applies when they are simultaneously in $\Lambda_{>1}$.

So, when $\Gamma\neq \Gamma_+$, we know there is a set $\Lambda_{\leq 1}$ or $\Lambda_{>1}$ with precisely $1$ element. Call that element $\lambda_1$ and call the other two elements $\lambda_2$ and $\lambda_3$. Then,$$L_f(z)=\frac{1-\lambda_1\lambda_2\lambda_3z}{1-z} \textrm{ and } L_{f_+}(z)=\frac{(1-\lambda_1z)(1-\lambda_1\lambda_2\lambda_3z)}{(1-z)(1-\lambda_2\lambda_3z)}.$$
By Theorem \ref{Nielsenzeta}, we have the following Nielsen zeta functions:

\begin{center}
\renewcommand{\arraystretch}{1.7}
\right)$, with $a,b$ and $c$ odd and $r,s,t \in \frac{1}{4}\Z \setminus\frac{1}{2}\Z$.

Since the eigenvalues are $a$ and $\pm \sqrt{bc}$ in these cases and because a similar argument as before applies, we know we have the following Nielsen zeta functions:

\begin{center}
\renewcommand{\arraystretch}{1.7}%
\right)$, with $a,b$ and $c$ odd, $r\in  \frac{1}{2}\Z$ and $s,t \in \frac{1}{4}\Z \setminus\frac{1}{2}\Z$.

Because of Corollary \ref{AnCor}, we know that $\Gamma=\Gamma_+$ in this case and a straightforward computation gives the following Nielsen zeta functions:
\begin{center}
\renewcommand{\arraystretch}{1.7}%
\right)$, with $r,s,t \in \R$.
In this case: $$N_f(z)=\frac{1}{1-z}.$$
\end{itemize}
\subsubsection{Other three-dimensional flat manifolds}
\begin{enumerate}
\item \underline{Generators:}
$$\left(%
\right) .$$
The holonomy group is isomorphic to $\Z_2$.\\

\underline{Possible affine homotopy lifts:}
\begin{itemize}
\item $\left(%
\right)$, with $a,b,c,d \in \Z$, $e$ odd, $r,s \in \frac{1}{2}\Z$ and $t\in \R$.

Let $\lambda_1,\lambda_2$ and $e$ be the eigenvalues of this matrix. Then, we have the following Nielsen zeta functions:

\begin{center}
\renewcommand{\arraystretch}{1.7}%
\right) .$$
The holonomy group is isomorphic to $\Z_2$.\\

\underline{Possible affine homotopy lifts:}
\begin{itemize}
\item $\left(%
\right)$, with $a,c$ and $e$ even, $b,d,f \in \Z$ and $r,s,t \in \R$ .

Let $\lambda_1,\lambda_2$ and $0$ be the eigenvalues of these matrices. Then the Nielsen zeta functions are:
\begin{center}
\renewcommand{\arraystretch}{1.7}%
\right)$, with $a$ odd, $b,d,e \in \Z$, $c$ even, $r,s \in \R$ and $t\in \frac{1}{2}\Z$.

Let $\lambda_1,\lambda_2$ and $e$ be the eigenvalues of this matrix. Then the Nielsen zeta functions are:

\begin{center}
\renewcommand{\arraystretch}{1.7}%
\right) .$$
The holonomy group is isomorphic to $\Z_2$.\\

\underline{Possible affine homotopy lifts:}
\begin{itemize}
\item $\left(%
\right)$, with $a,c$ and $e$ even, $b,d,f \in \Z$, $r,s,t \in \R$.

Let $\lambda_1,\lambda_2$ and $0$ be the eigenvalues of these matrices. Then the Nielsen zeta functions are:
\begin{center}
\renewcommand{\arraystretch}{1.7}%
\right)$, with $a$ odd, $c$ even, $b,d,e \in \Z$, $r\in \R$ and $s-t\in  \Z$.

Let $\lambda_1,\lambda_2$ and $d-e=f$ be the eigenvalues of these matrices. Then the Nielsen zeta functions are:

\begin{center}
\renewcommand{\arraystretch}{1.7}%
\right) .$$
The holonomy group is isomorphic to $\Z_2\oplus \Z_2$.\\

\underline{Possible affine homotopy lifts:}
\begin{itemize}
\item $\left(%
\right)$, with $a,c$ odd, $b \in \Z$, $r,s\in \frac{1}{2}\Z$ and $t\in \R$.

We distinguish three cases. If $|a|,|b|\leq 1$, $\rho_{>1}$ will be the identity representation and therefore we can only have the following Nielsen zeta functions:

\begin{center}
\renewcommand{\arraystretch}{1.7}\begin{tabular}{ c|c|c|c|c }
\multicolumn{1}{l}{ }
 &  \multicolumn{1}{|c|}{$p$ even, $n$ even}
 & \multicolumn{1}{|c|}{$p$ even, $n$ odd}& \multicolumn{1}{|c|}{$p$ odd, $n$ even}& \multicolumn{1}{|c}{$p$ odd, $n$ odd} \\
\cline{1-5}
$\Gamma= \Gamma_+$ & $\ds\frac{1-cz}{1-z}$ & $\ds\frac{1+z}{1+cz}$ &$\ds\frac{1-z}{1-cz}$ & $\ds\frac{1+cz}{1+z}$\\[1ex]
\end{tabular}
 \end{center}
 
If $|a|,|b|>1$, then we have the following Nielsen zeta functions:

\begin{center}
\renewcommand{\arraystretch}{1.7}\begin{tabular}{ c|c|c|c|c }
\multicolumn{1}{l}{ }
 &  \multicolumn{1}{|c|}{$p$ even, $n$ even}
 & \multicolumn{1}{|c|}{$p$ even, $n$ odd}& \multicolumn{1}{|c|}{$p$ odd, $n$ even}& \multicolumn{1}{|c}{$p$ odd, $n$ odd} \\
\cline{1-5}
$\Gamma\neq \Gamma_+$ &$\ds\frac{1-abcz}{1-abz}$ & $\ds\frac{1+abz}{1+abcz}$ &$\ds\frac{1-abz}{1-abcz}$ & $\ds\frac{1+abcz}{1+abz}$\\[1ex]
\end{tabular}
 \end{center}

If $|a|>1$ and $|b|\leq 1$ or $|a|\leq 1$ and $|b|>1$, let us denote $\lambda$ as the element of $\{a,b\}$ with the largest modulus. Then the Nielsen zeta functions are the following:

  \begin{center}
  \renewcommand{\arraystretch}{1.7}%
\right) .$$
The holonomy group is isomorphic to $\Z_2\oplus \Z_2$.\\

\underline{Possible affine homotopy lifts:}
\begin{itemize}
\item $\left(%
\right)$, with $a,b$ and $c$ odd, $r,s\in \frac{1}{2}\Z$ and $t\in \R$.

We distinguish three cases. 
If $|a|,|b|\leq 1$, then we have the following Nielsen zeta functions:

\begin{center}
\renewcommand{\arraystretch}{1.7}\begin{tabular}{ c|c|c|c|c }
\multicolumn{1}{l}{ }
 &  \multicolumn{1}{|c|}{$p$ even, $n$ even}
 & \multicolumn{1}{|c|}{$p$ even, $n$ odd}& \multicolumn{1}{|c|}{$p$ odd, $n$ even}& \multicolumn{1}{|c}{$p$ odd, $n$ odd} \\
\cline{1-5}
$\Gamma= \Gamma_+$ & $\ds\frac{1-cz}{1-z}$ & $\ds\frac{1+z}{1+cz}$ &$\ds\frac{1-z}{1-cz}$ & $\ds\frac{1+cz}{1+z}$\\[1ex]
\end{tabular}
 \end{center}
 
If $|a|,|b|>1$, then we have the following Nielsen zeta functions:

\begin{center}
\renewcommand{\arraystretch}{1.7}\begin{tabular}{ c|c|c|c|c }
\multicolumn{1}{l}{ }
 &  \multicolumn{1}{|c|}{$p$ even, $n$ even}
 & \multicolumn{1}{|c|}{$p$ even, $n$ odd}& \multicolumn{1}{|c|}{$p$ odd, $n$ even}& \multicolumn{1}{|c}{$p$ odd, $n$ odd} \\
\cline{1-5}
$\Gamma\neq \Gamma_+$ &$\ds\frac{1-abcz}{1-abz}$ & $\ds\frac{1+abz}{1+abcz}$ &$\ds\frac{1-abz}{1-abcz}$ & $\ds\frac{1+abcz}{1+abz}$\\[1ex]
\end{tabular}
 \end{center}

If $|a|>1$ and $|b|\leq 1$ or $|a|\leq 1$ and $|b|>1$, let us denote $\lambda$ as the element of $\{a,b\}$ with the largest modulus. Then the Nielsen zeta functions are the following:

  \begin{center}
  \renewcommand{\arraystretch}{1.7}%
\right).$$
 The holonomy group is isomorphic to $\Z_4$.\\
 
\underline{Possible affine homotopy lifts:}
 \begin{itemize}
 \item $\left(%
\right)$, with $a,b \in \Z$, $c\equiv 1 \mod 4$, $r,s\in \frac{1}{2}\Z\setminus \Z$ and $t\in \R$.

Because of Theorem \ref{AnCycl}, we know that $\Gamma=\Gamma_+$ for all possible maps on this infra-nilmanifold. With some clever computations we have the following Nielsen zeta functions:

\begin{center}
\renewcommand{\arraystretch}{1.5}%
\right).$$
 The holonomy group is isomorphic to $\Z_3$.\\

 \underline{Possible affine homotopy lifts:}
 \begin{itemize}
 \item $\left(%
\right)$, with $a,b \in \Z$, $c \equiv 1 \mod 3$, $r \in \frac{2}{3}+\Z$, $s \in \frac{1}{3}+\Z$ and $t\in \R$.
 
 All maps on this infra-nilmanifold will satisfy the Anosov relation because of Theorem \ref{AnCycl} or Theorem \ref{An2P}. Hence, $\Gamma=\Gamma_+$ and the Nielsen zeta functions can be found in the following table:
 
  \begin{center}
    \renewcommand{\arraystretch}{2}%
\right).$$
 The holonomy group is isomorphic to $\Z_6$.\\
 
\underline{Possible affine homotopy lifts:}
 \begin{itemize}
 \item $\left(%
\right)$, with $a,b \in \Z$, $c \equiv 1 \mod 6$, $r,s\in \Z$ and $t\in \R$.

Again, $\Gamma=\Gamma_+$, for every possible map, because of Theorem \ref{AnCycl}.
\begin{center}
    \renewcommand{\arraystretch}{2}\begin{tabular}{ c|c|c }
   \multicolumn{1}{c|}{$\Gamma= \Gamma_+$}& \multicolumn{1}{|c|}{$p$ even}
   & \multicolumn{1}{|c}{$p$ odd} \\
  \cline{1-3}
  $n$ even & $\ds\frac{(1-cz)(1-(a^2+b^2+ab)cz)}{(1-z)(1-(a^2+b^2+ab)z)}$ &$\ds\frac{(1-z)(1-(a^2+b^2+ab)z)}{(1-cz)(1-(a^2+b^2+ab)cz)}$ \\[1ex]
  \cline{1-3}
  $n$ odd & $\ds\frac{(1+z)(1+(a^2+b^2+ab)z)}{(1+cz)(1+(a^2+b^2+ab)cz)}$ & $\ds\frac{(1+cz)(1+(a^2+b^2+ab)cz)}{(1+z)(1+(a^2+b^2+ab)z)}$
  
  \end{tabular}%
   \end{center}

 \end{itemize}\end{enumerate}
\subsection{The two-step nilpotent case}
All possible almost-Bieberbach groups can be found in \cite{deki96-1}, where they are already written down in the desired form, embedded in $\aff(\R^3)$, with the right lattice. We will use the same names for our groups as in \cite{deki96-1}.

\medskip

We will now describe all possible maps on infra-nilmanifolds modeled on the Heisenberg group $H$. It will suffice to list all possibilities for $\D_\ast$ and $\dd=h(r,s,t)$, for which $(\dd,\D)$ is the affine homotopy lift of the map $f:\Gamma\backslash H\to \Gamma \backslash H$. Indeed, because $H$ is a simply connected, connected, nilpotent Lie group, the maps $\log$ and $\exp$ are diffeomorphisms and it suffices to describe $\D_\ast:\lieH\to \lieH$ in order to fully understand the endomorphism $\D$ on $H$. Since $\D_\ast$ is the map we need to know in order to calculate Lefschetz and Nielsen numbers, this is the desired way of listing all possible maps from our point of view.

\subsubsection{Nilmanifolds}
\begin{enumerate}
\item Type I, $k>0$: $$\Gamma=\{a,b,c\|[b,a]=c^k,[c,a]=1,[c,b]=1\}$$

\underline{Generators:}
 $$\psi(a)=\left(
\right)$, with $a,b,c,d \in \Z$ and $x,y , r,s,t\in \R$, with
$\ds x+k(as-cr+\frac{ac}{2})\in \Z$ and $\ds y+k(bs-dr+\frac{bd}{2})\in \Z$.
 
 Let $\lambda_1, \lambda_2$ and $\lambda_1\lambda_2$ be the eigenvalues of $\D_\ast$. Because of Theorem \ref{AnNil}, we know that $\Gamma=\Gamma_+$. We then find the following Nielsen zeta functions:

 \begin{center} \renewcommand{\arraystretch}{2}\begin{tabular}{ c|c|c }
  \multicolumn{1}{c|}{$\Gamma=\Gamma_+$}& \multicolumn{1}{|c|}{$p$ even}
  & \multicolumn{1}{|c}{$p$ odd} \\
 \cline{1-3}
 $n$ even &  $\ds\frac{(1-\lambda_1z)(1-\lambda_2z)(1-\lambda_1^2\lambda_2^2z)}{(1-z)(1-\lambda_1^2\lambda_2z)(1-\lambda_1\lambda_2^2z)} $ &$\ds\frac{(1-z)(1-\lambda_1^2\lambda_2z)(1-\lambda_1\lambda_2^2z)}{(1-\lambda_1z)(1-\lambda_2z)(1-\lambda_1^2\lambda_2^2z)} $  \\[1ex]
 \cline{1-3}
 $n$ odd & $\ds\frac{(1+z)(1+\lambda_1^2\lambda_2z)(1+\lambda_1\lambda_2^2z)}{(1+\lambda_1z)(1+\lambda_2z)(1+\lambda_1^2\lambda_2^2z)} $ & $\ds\frac{(1+\lambda_1z)(1+\lambda_2z)(1+\lambda_1^2\lambda_2^2z)}{(1+z)(1+\lambda_1^2\lambda_2z)(1+\lambda_1\lambda_2^2z)} $
 
 \end{tabular}
  \end{center}

 \end{itemize}
\end{enumerate}
 \subsubsection{Other infra-nilmanifolds}
From now on, we will always assume $\psi(a), \psi(b)$ and $\psi(c)$ to be part of the generators of $\psi(\Gamma)$.\\

\begin{enumerate}[resume]
\item Type II, $k>0$, $k\equiv 0 \mod 2$: $$\Gamma=\{a,b,c, \alpha \|[b,a]=c^k,[c,a]=1,[c,b]=1, \alpha c=c \alpha, \alpha b= b^{-1} \alpha, \alpha a=a^{-1}\alpha, \alpha^2=c\}$$
 $$\psi(\alpha)=\left(%
\right) .$$
The holonomy group is therefore isomorphic to $\Z_2$.\\

\underline{Possibilities for $\D_\ast$:}
\begin{itemize}
\item $\left(%
\right)$, with $a,b,c,d \in \Z$, $r,s\in \frac{1}{2}\Z$ and $t\in \R$, with $ad-bc$ odd.
 
 Let $\lambda_1, \lambda_2$ and $\lambda_1\lambda_2$ be the eigenvalues of $\D_\ast$. We then find the following Nielsen zeta functions:
  \begin{center}
   \renewcommand{\arraystretch}{1.9}%
\right)$, with $r,s,t\in  \R$.
The Nielsen zeta function then equals$$\frac{1}{1-z}.$$
\end{itemize}

\item Type IV, $k>0$, $k\equiv 0 \mod 2$:$$\Gamma=\{a,b,c, \alpha \|[b,a]=c^k,[c,a]=1,[c,b]=1, \alpha c=c^{-1} \alpha, \alpha b= b^{-1} \alpha c^{\frac{-k}{2}}, \alpha a=a\alpha, \alpha^2=a\}$$
 $$\psi(\alpha)=\left(%
\right) .$$
The holonomy group is therefore isomorphic to $\Z_2$.\\

\underline{Possibilities for $\D_\ast$:}
\begin{itemize}
\item $\left(%
\right)$, with $a$ odd, $b \in \Z$, $s\in \frac{1}{2}\Z$ and $y,r,t\in \R$, with $y-kbr\in \Z$ and $2krs-\frac{kra}{2}-2r\in \Z$.
 
 We then find the following Nielsen zeta functions:
  \begin{center}
   \renewcommand{\arraystretch}{1.9}%
\right)$, with $a,b$ even, $x,r,s,t\in  \R$ and  $x-k(br+as)\in 2\Z$.
 
Because of Corollary \ref{AnCor}, we know $\Gamma=\Gamma_+$ and therefore, we can only find the following Nielsen zeta functions:
  \begin{center}
  \renewcommand{\arraystretch}{1.7}%
\right) .$$
The holonomy group is therefore isomorphic to $\Z_2 \oplus\Z_2$.\\

\underline{Possibilities for $\D_\ast$:}

\begin{itemize}
\item $\left(%
\right)$, with $a,b$ odd, $r,s\in \frac{1}{2}\Z$ and $t\in \frac{1}{4}\Z\setminus \frac{1}{2}\Z$.

 Note that the eigenvalues are $-ab$ and $\pm\sqrt{ab}$. Hence, $\Gamma=\Gamma_+$ and we have the following Nielsen zeta functions:
 
  \begin{center}
   \renewcommand{\arraystretch}{1.9}%
\right)$, with $a,b$ odd and $r,s,t\in \frac{1}{2}\Z$.

The eigenvalues of this matrix are $a,b$ and $ab$. Since $a$ and $b$ are odd integers and since $\Gamma\backslash H$ is an oriented infra-nilmanifold, it is only possible that $\Gamma\neq \Gamma_+$ if precisely one element of the set $\{a,b\}$ has modulus $1$. Call this element $\lambda$. Then we have the following Nielsen zeta functions: 
  
   \begin{center}
    \renewcommand{\arraystretch}{1.9}%
\right)$, with $r,s,t\in \R$.

The Nielsen zeta function is$$\frac{1}{1-z}.$$

\end{itemize}

\item Type X, $k>0$, $k\equiv 0 \mod 2$:$$\Gamma=\{a,b,c, \alpha, \beta \|[b,a]=c^k,[c,a]=1,[c,b]=1, \alpha c=c \alpha, \alpha b= a^{-1} \alpha, \alpha a=b \alpha, \alpha^4=c\}$$
$$\psi(\alpha)=\left(%
\right) .$$

The holonomy group is therefore isomorphic to $\Z_4$.\\

\underline{Possibilities for $\D_\ast$:}

\begin{itemize}
\item $\left(%
\right)$, with $a,b\in \Z$, $a+b$ odd, $r,s\in \frac{1}{2}\Z\backslash \Z$, $t\in \R$, $k\equiv 0 \mod 4$.
Note that $\Gamma=\Gamma_+$, because of Theorem \ref{AnCycl}. Hence, the following table consists of all possible Nielsen zeta functions:
  \begin{center}
     \renewcommand{\arraystretch}{1.9}%
\right)$, with $r,s,t\in \R$.

The Nielsen zeta function is$$\frac{1}{1-z}.$$
\end{itemize}
\item Type X, $k>0$, $k\equiv 0 \mod 4$:$$\Gamma=\{a,b,c, \alpha, \beta \|[b,a]=c^k,[c,a]=1,[c,b]=1, \alpha c=c \alpha, \alpha b= a^{-1} \alpha, \alpha a=b \alpha, \alpha^4=c^3\}$$
$$\psi(\alpha)=\left(%
\right) .$$

The holonomy group is therefore isomorphic to $\Z_4$.\\

\underline{Possibilities for $\D_\ast$:}

\begin{itemize}
\item $\left(%
\right)$, with $r,s,t\in \R$.

The Nielsen zeta function is$$\frac{1}{1-z}.$$
\end{itemize}

\item Type XIII, $k>0$, $k\equiv 0 \mod 3$:$$\Gamma=\{a,b,c, \alpha, \beta \|[b,a]=c^k,[c,a]=1,[c,b]=1, \alpha c=c \alpha, \alpha b= a^{-1}b^{-1} \alpha, \alpha a=b \alpha, \alpha^3=c\}$$
$$\psi(\alpha)=\left(%
\right) .$$

The holonomy group is therefore isomorphic to $\Z_3$.\\

\underline{Possibilities for $\D_\ast$:}

\begin{itemize}
\item $\left(%
\right)$,\\ with $a,b\in \Z$, $a\not \equiv b \mod{3}$, and $r,s,t\in \R$, such that $r+s,2r-s \in \Z$. 

By Theorem \ref{AnCycl} or Theorem \ref{An2P}, we know that $\Gamma=\Gamma_+$ for every possible map on this infra-nilmanifold.
  
  \begin{center}
     \renewcommand{\arraystretch}{1.5}%
\right)$, with $r,s,t\in \R$.

The Nielsen zeta function is$$\frac{1}{1-z}.$$
\end{itemize}

\item Type XIII, $k>0$, $k\equiv 0 \mod 3$:$$\Gamma=\{a,b,c, \alpha, \beta \|[b,a]=c^k,[c,a]=1,[c,b]=1, \alpha c=c \alpha, \alpha b= a^{-1}b^{-1} \alpha, \alpha a=b \alpha, \alpha^3=c^2\}$$
$$\psi(\alpha)=\left(%
\right) .$$

The holonomy group is therefore isomorphic to $\Z_3$.\\

\underline{Possibilities for $\D_\ast$:}

\begin{itemize}
\item $\left(%
\right)$, with $r,s,t\in \R$.

The Nielsen zeta function is$$\frac{1}{1-z}.$$
\end{itemize}

\item Type XIII, $k>0$, $k\not \equiv 0 \mod 3$:$$\Gamma=\{a,b,c, \alpha, \beta \|[b,a]=c^k,[c,a]=1,[c,b]=1, \alpha c=c \alpha, \alpha b= a^{-1}b^{-1} \alpha, \alpha a=b \alpha c, \alpha^3=c\}$$
$$\psi(\alpha)=\left(%
\right) .$$
The holonomy group is therefore isomorphic to $\Z_3$.\\

\underline{Possibilities for $\D_\ast$:}
\begin{itemize}
\item $\left(%
\right)$, with $r,s,t\in \R$.

The Nielsen zeta function is$$\frac{1}{1-z}.$$
\end{itemize}

\item Type XVI, $k>0$, $k \equiv 0 \textrm{ or } 4 \mod 6$:$$\Gamma=\{a,b,c, \alpha, \beta \|[b,a]=c^k,[c,a]=1,[c,b]=1, \alpha c=c \alpha, \alpha b= a^{-1} \alpha, \alpha a=ab \alpha, \alpha^6=c\}$$
$$\psi(\alpha)=\left(%
\right) .$$
The holonomy group is therefore isomorphic to $\Z_6$.\\

\underline{Possibilities for $\D_\ast$:}

\begin{itemize}
\item $\left(%
\right)$,\\ with $a,b\in \Z$, $a^2+b^2+ab \equiv 1 \mod{6}$, $r,s \in \Z$, $t\in \R$.  

Again, Theorem \ref{AnCycl} tells us that the Anosov relation will hold for every map on this infra-nilmanifold.
  
  \begin{center}
     \renewcommand{\arraystretch}{1.5}%
\right)$, with $r,s,t\in \R$.
The Nielsen zeta function is$$\frac{1}{1-z}.$$
\end{itemize}

\item Type XVI, $k>0$, $k \equiv 0 \textrm{ or } 2\mod 6$:$$\Gamma=\{a,b,c, \alpha, \beta \|[b,a]=c^k,[c,a]=1,[c,b]=1, \alpha c=c \alpha, \alpha b= a^{-1} \alpha, \alpha a=ab \alpha, \alpha^6=c^5\}$$
$$\psi(\alpha)=\left(%
\right) .$$
The holonomy group is therefore isomorphic to $\Z_6$.\\

\underline{Possibilities for $\D_\ast$:}

\begin{itemize}
\item $\left(%
\right)$, with $r,s,t\in \R$.
The Nielsen zeta function is$$\frac{1}{1-z}.$$
\end{itemize}
\end{enumerate}
\bibliography{G:/algebra/ref}
\bibliographystyle{G:/algebra/ref}
\end{document}